\documentclass[oneside]{amsart}
\newcommand{\formatswitch}{preprint}

\usepackage{amssymb}
\usepackage{ifthen}
\usepackage{verbatim}
\usepackage{epsfig}
\usepackage[all]{xy}

\newcommand{\tref}[1]{(\ref{#1})}

\ifthenelse{\equal{\formatswitch}{big}}{
\setlength{\textwidth}{432pt}
\setlength{\oddsidemargin}{18.8775pt}

\setcounter{tocdepth}{1}
}{
\ifthenelse{\equal{\formatswitch}{paper}}{

\setcounter{tocdepth}{1}
}{
\setcounter{tocdepth}{2}
}
}
\DeclareMathAlphabet\EuScript{U}{eus}{m}{n}
\DeclareMathAlphabet\EuScriptb{U}{eus}{b}{n}

\newcommand{\sscr}[1]{\EuScript{#1}}

\newcommand{\claimenum}{\renewcommand{\theenumi}{\alph{enumi}}
 \renewcommand{\labelenumi}{\textit{(\theenumi)}}
 \renewcommand{\theenumii}{\roman{enumii}}
 \renewcommand{\labelenumii}{\textit{(\theenumii)}}
 \begin{enumerate}}
\newcommand{\claimenumend}{\end{enumerate}}

\newcommand{\romanenum}{\renewcommand{\theenumi}{\roman{enumi}}
 \renewcommand{\labelenumi}{\textit{(\theenumi)}}
 \renewcommand{\theenumii}{\alph{enumii}}
 \renewcommand{\labelenumii}{\textit{(\theenumii)}}
 \begin{enumerate}}
\newcommand{\romanenumend}{\end{enumerate}}

\newtheorem{dummy}{realdumb}[section]
\newtheorem{thm}{Theorem}

\newtheorem{lemma}[dummy]{Lemma}
\newtheorem{prop}[dummy]{Proposition}
{\theoremstyle{definition} }

{\theoremstyle{definition} }

\renewcommand{\text}{\mathrm}

\newcommand{\strutdepth}{\dp\strutbox}
\newcommand{\marginalnote}[1]
   {\strut\vadjust{\kern-\strutdepth\domarginalnote{#1}}}
\newcommand{\domarginalnote}[1]{\vtop to \strutdepth{
  \baselineskip\strutdepth
   \vss\llap{ #1\ \ }\null}}  
\newcounter{showlabelflag}
\newcounter{makelabelflag}
\newcommand{\showlabels}{\setcounter{showlabelflag}{1}}

\newcommand{\makelabels}{\setcounter{makelabelflag}{1}}

\newcommand{\mylabel}[1]{
  \ifthenelse{\value{makelabelflag}=1}
    {\label{#1}}{}
  \ifthenelse{\value{showlabelflag}=1}
    {\marginpar{#1}}{}\relax}

\newcommand{\R}{{\mathbf R}}

\newcommand{\Z}{{\mathbf Z}}

\ifthenelse{\equal{\formatswitch}{draft}}{
\showlabels
}{\relax}

\newcommand{\scr}{\sscr}

\newcommand{\mymargin}[1]{
  \ifthenelse{\value{showlabelflag}=1}
    {\marginpar{#1}}{}\relax}

\newcounter{enumo}

\newcommand{\RRsh}{\kern -1 pt \Rsh}

\newcounter{keepitemnum}

\newcounter{keepitemnumm}

\begin{document}

\bibliographystyle{amsplain}
\title[Rank 2 free limits]{The Free Group of Rank 2 is a Limit \\
of Thompson's Group \(F\)}
\thanks{AMS
Classification (2000): primary 20F69, secondary 20F05, 20F38.}
\author{Matthew G. Brin}
\date{April 20, 2009}
\CompileMatrices


\makelabels
\maketitle


\begin{abstract}{
We show that the free group of rank 2 is a limit of 
2-markings of Thompson's group \(F\) in the space of all 2-marked 
groups. More specifically, we find a sequence of generating pairs 
for \(F\) so that as one goes out the sequence, the length of the 
shortest relation satisfied by the generating pair goes to infinity. 
}\end{abstract}

\section{Introduction}

From \cite{zarzycki}, a \(k\)-marked group is a pair \((G, S)\)
where \(S\) is an ordered \(k\)-tuple of generators (the
\(k\)-marking) of the group \(G\).  An isomorphism of marked groups
must preserve the markings.  With \(\scr{G}_k\) the set of all
isomorphism classes of \(k\)-marked groups, one says that two
elements of \(\scr{G}_k\) are no more than \(e^{-R}\) apart if they
have satisfy the same relations of length no longer than \(R\).
This gives a metric on \(\scr{G}_k\), and a group \(L\) is a
\(G\)-limit if it is a limit in \(\scr{G}_k\) of a sequence of
marked groups each of which is isomorphic as an unmarked group to
\(G\).

Marked groups and their limits are defined and studied in
\cite{MR2151593} where, among other things, it is shown that
\(\scr{G}_k\) is compact. Limits of marked groups extend the notion
of limit groups of \cite{MR1863735}.

Groups that cannot be obtained as a limit of Thompson's group \(F\)
are studied in \cite{zarzycki}.

Recently Akhmedov, Stein and Taback \cite{AST}
have announced that the free
group on \(k\) generators is a limit of \(k\)-markings of Thompson's group
\(F\) when \(k \ge3\).  The purpose of this paper is to prove the following.  

\begin{thm}\mylabel{THM} The free group on 2 generators
is a limit of 2-markings of \(F\).  \end{thm}

See \cite{CFP} for an introduction to Thompson's group \(F\).

There is a closely related notion of free-like.  Essentially a group
\(G\) is \(k\)-free-like if both (a) the free group on \(k\)
generators can be obtained as a \(G\)-limit, and (b) the group \(G\)
is ``uniformly non-amenable'' with respect to the markings used in
the limit.  See \cite{Sapir+olshanskii:free-like} for the full
definition, examples and discussion of the free-like property and
related concepts.

The notion of a free group occuring as a \(G\)-limit can be easily
restated.  The free group of rank \(k\) is a \(G\)-limit if there is
a sequence of \(k\)-tuples \(S_n\) in \(G\) so that (a) each \(S_n\)
generates \(G\), and (b) for each \(n\), no reduced word of length
at most \(n\) in the elements of \(S_n\) and their inverses
represents the trivial element of \(G\).  Alternatively for (b), one
can say that the homomorphism from the free group on \(k\)
generators taking the generators of the free group to the elements
of \(S_n\) embeds the \(n\)-ball of the free group in \(G\).

Our proof combines two tendencies in \(F\).  First, it is not hard
to find elements in \(F\) that satisfy no short relations.  Second
it is rather ``easy'' to generate \(F\).  The word ``easy'' is in
quotes because while it might not be hard to pick out sets of
elements that generate \(F\), the calculations that show that they
generate might be complicated.  This is discussed more in the body
of the paper.

The technique that we use to embed large balls of the free group in
\(F\) is taken from \cite{brin+squier}.  An alternative technique is
found in \cite{esyp}.  However it is not clear that the alternative
technique lends itself as well to building generators.

In Section \ref{Outline} we give the outline.  In Section
\ref{Details} we give the details that show that the outline is
correct.  The outline is the more important part of the paper, and
the details should be read only by those that wish to check for
correctness.

Background for this paper would include \cite{CFP} for its general
introduction to the group \(F\), for the normal form of an element
in terms of the infinite generating set, and for the description in
Section 1 of \cite{CFP} of the ``rectangular diagrams'' of Thurston
which are aids to calculation.  Rectangular diagrams will be used
extensively in Section \ref{Details}.  We use the representation of
\(F\) as a group of right acting homeomorphisms of the non-negative
real numbers.  A discussion close to this view is found in Section 2
of \cite{brin+fer} and to some extent in \cite{brown:finiteprop}.

We would like to thank Mark Sapir for supplying the question
answered by Theorem \ref{THM}.

\section{Outline}\mylabel{Outline}

The model of \(F\) that we work with is the model on \[\R_{\ge0} =
\{t\in \R\mid t\ge0\}\] in which the generators are the self 
homeomorphisms \(x_i\), \(i\ge0\), of \(\R_{\ge0}\) 
operating on the right defined by
\[tx_i = \begin{cases} t, &t\le i, \\
i+2(t-i), &i\le t\le i+1, \\
t+1, &t\ge i+1. \end{cases}\]

The \(x_i\) satisfy the usual relations \(x_jx_i = x_ix_{j+1}\) whenever 
\(i<j\), and it is known that \(x_0\) and \(x_1\) generate \(F\).

Our method will be to modify \(x_0\) and \(x_1\) "slightly" so that (I)
they still generate, and (II) they don't satisfy short relations.  The 
slightness of the modification will mean that in a region large
enough to be useful, the modifications agree with the originals.

\subsection{Getting generators}

Showing (I), that the modifications still generate, will need the
more intricate argument.  I learned the following use of the
subgroups \(F_{[a,b]}\) from Collin Bleak and Bronlyn Wassink.

Let \([a,b]\) be a closed interval, and let \(F_{[a,b]}\) be all the
elements in \(F\) whose support is contained in \([a,b]\).  The
{\itshape support} of an \(f\in F\) is the set \(\{t\in \R_{\ge0}
\mid tf\ne t\}\).

We will only refer to \(F_{[a,b]}\)
when \(0\le a<b\) and both \(a\) and \(b\) are dyadic (that is, of
the form \(m/2^n\) with \(m\) and \(n\) in \(\Z\)).  It is standard
that under these restrictions \(F_{[a,b]}\) is isomorphic to \(F\).  

When \(a\) and \(b\) are two consecutive integers, it is very 
easy to write down generators for \(F_{[a,b]}\).  For \(a=0\) and \(b=1\), 
we have the usual model of \(F\) on \([0,1]\) and typical generators 
for \(F_{[0,1]}\) are
\(y_0\) and \(z_0\) defined by
\[\begin{split} 
ty_0 &= \begin{cases} 2t, &0\le t\le \frac14, \\
t+\frac14, &\frac14\le t\le \frac12, \\
1-\frac12(1-t), &\frac34 \le t\le 1, \end{cases} \\
tz_0 &= \begin{cases} 2t, &0\le t\le \frac18, \\
t+\frac18, &\frac18\le t\le \frac14, \\
\frac12-\frac12(\frac12-t), &\frac38 \le t\le \frac12, \\
t, &\frac12 \le t\le 1. \end{cases} 
\end{split}\]

It is standard and an easy exercise that 
\[\begin{split} y_0 &= x_0^2 x_1^{-1}x_0^{-1}, \\
z_0 &= x_0^3 x_1^{-1} x_0^{-2}.\end{split}\]

To generate \(F_{[i,i+1]}\) for a positive integer \(i\), we use 
\(y_i\) and \(z_i\) where 
\mymargin{YZDef}\begin{equation}\label{YZDef}
\begin{split} 
y_i &= x_i^2 x_{i+1}^{-1}x_i^{-1}, \\
z_i &= x_i^3 x_{i+1}^{-1} x_i^{-2}.
\end{split}
\end{equation}

We will also use the elements 
\mymargin{WIDef}\begin{equation}\label{WIDef}
w_i=x_ix_{i+1}^{-1}
\end{equation}
and it is also standard and easy exercise that \(w_i\) and \(y_i\)
generate \(F_{[i,i+2]}\) in a manner "identical" to the manner in
which \(y_i\) and \(z_i\) generate \(F_{[i,i+1]}\).

\begin{lemma}\mylabel{GenLem}  Let \(X_0\) and 
\(X_1\) generate a subgroup \(G\) of \(F\) and assume the following hold.
{\claimenum
\item The support of \(X_0x_0^{-1}\) is contained in a compact subset
of \((0,\infty)\).
\item The support of \(X_1x_1^{-1}\) is contained in a compact subset
of \((0,\infty)\).
\item The group \(G\) generated by \(X_0\) and \(X_1\)
contains \(w_1\) and \(y_1\).
\item The translates of the open interval \((1,3)\) under \(G\) cover all 
of \((0,\infty)\).
\claimenumend}
Then  \(G=F\).  \end{lemma}

\begin{proof} From (c), we know that \(G\) contains \(F_{[1,3]}\),
from (d) we know that \(G\) contains all elements of \(F\) whose
support is contained in a compact subset of \((0,\infty)\), and from (a)
and (b) we know that \(G\) contains \(X_0x_0^{-1}\) and
\(X_1x_1^{-1}\) and thus \(x_0\) and \(x_1\).   Thus \(G\)
contains all of \(F\).  \end{proof}

\subsection{Avoiding relations}

Showing (II), that our chosen generators don't satisfy short relations,
will be done as in \cite{brin+squier}.  We will take homeomorphisms on the
circle that generate a free group and lift these to the real line.  The 
lifts will not be elements of \(F\), but approximations with supports on 
compact subsets will be elements of \(F\).  If the compact subsets are long
enough, then as shown in \cite{brin+squier} the only relations that the 
approximations satisfy must be long.  Let us refer to these approximations
as "almost free" elements.

We will need to know more about these "almost free elements" than
just the fact that they only satisfy long relations.  We will also
need to know some specific relations that they do satisfy.  As is
typically the case with \(F\), these relations will be commutators
of certain words in the almost free elements.  These relations will
be used to build our generating set.

\subsection{Getting generators, revisited}

We can now give better descriptions of our generators \(X_0\) and
\(X_1\).

We will work with two intervals.  On one interval \([0,b-5)\) for
some \(b>5\) that we will choose, \(X_i\) will agree with \(x_i\)
for \(i=0,1\).  (The strange way of giving the upper limit of the
interval is to make things convenient later.)  On a second interval
\((b-5,d)\) for a \(d>b\) that will depend on \(n\), \(X_i\) will
agree with an element \(g_i\) for \(i=0,1\), (also depending on
\(n\)) for which it is known that the \(g_i\) satisfy no 
relations shorter than \(n\) but which do satisfy certain specific
relations that are longer than \(n\).

We then consider words in the \(X_i\).  We will use capital letters
to denote such words.  For example 
\[C=X_0^2 X_1^2 X_0^{-2} X_1^{-2}\] will be one word.  The corresponding
lower case letter will denote the corresponding word in the \(x_i\).
Thus \[c=x_0^2 x_1^2 x_0^{-2} x_1^{-2}\] denotes the word corrdesponding
to \(C\).

For certain words (let \(W\) denote such a word for this discussion)
in the \(X_i\), the corresponding word (\(w\) in this example) in
the \(x_i\) will satisfy \(W=w\) in \(F\) because
\begin{enumerate} 
\item \(W\) will be trivial on \((b-5,d)\), 
\item the \(X_i\) agree with the \(x_i\) on \((0,b-5)\), and 
\item the word \(w\) in the \(x_i\) has support in \([0,b-5)\).  
\end{enumerate}

\subsection{Avoiding relations, revisited}

We now describe the \(X_i\).  Our first basic building block will be the
piece of function below on the left and its inverse below on the
right.
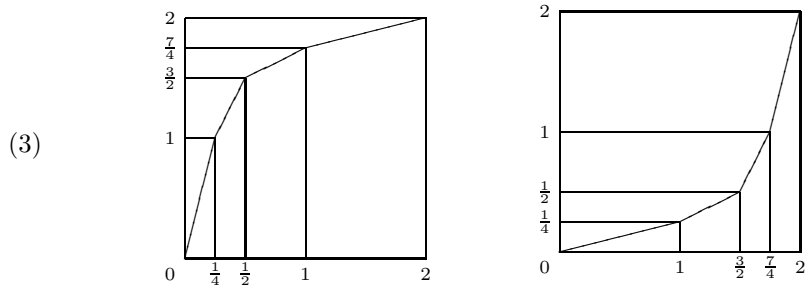
\mymargin{BasicBlocks}\begin{equation}\label{BasicBlocks}
\begin{split}
\xy
(0,0); (0,32)**@{-}; (32,32)**@{-}; (32,0)**@{-}; (0,0)**@{-};
(4,0); (4,16)**@{-}; (0,16)**@{-};
(8,0); (8,24)**@{-};  (0,24)**@{-};
(16,0); (16,28)**@{-}; (0,28)**@{-};
(0,0); (4,16)**@{-}; (8,24)**@{-}; (16,28)**@{-}; (32,32)**@{-};
(-2,-2)*{\scriptstyle 0};
(4,-2)*{\scriptstyle \frac14};
(8,-2)*{\scriptstyle \frac12};
(16,-2)*{\scriptstyle 1};
(32,-2)*{\scriptstyle 2};
(-2,16)*{\scriptstyle 1};
(-2,24)*{\scriptstyle \frac32};
(-2,28)*{\scriptstyle \frac74};
(-2,32)*{\scriptstyle 2};
\endxy
\quad
\quad
\quad
\quad
\xy
(0,0); (0,32)**@{-}; (32,32)**@{-}; (32,0)**@{-}; (0,0)**@{-};
(0,4); (16,4)**@{-}; (16,0)**@{-};
(0,8); (24,8)**@{-};  (24,0)**@{-};
(0,16); (28,16)**@{-}; (28,0)**@{-};
(0,0); (16,4)**@{-}; (24,8)**@{-}; (28,16)**@{-}; (32,32)**@{-};
(-2,-2)*{\scriptstyle 0};
(-2,4)*{\scriptstyle \frac14};
(-2,8)*{\scriptstyle \frac12};
(-2,16)*{\scriptstyle 1};
(-2,32)*{\scriptstyle 2};
(16,-2)*{\scriptstyle 1};
(24,-2)*{\scriptstyle \frac32};
(28,-2)*{\scriptstyle \frac74};
(32,-2)*{\scriptstyle 2};
\endxy
\end{split}
\end{equation}

The reader can verify that the function on the left is \(w_0^2\), a
fact that will be convenient for notation but that is not terribly
important otherwise.  The most important property that we need is
that \(\frac12\) is taken to \(\frac32\).  That is, all points in
\([0,2]\) not within \(\frac12\) of the left fixed point are carried
to within \(\frac12\) of the right fixed point.  Also important is
that no point moves more than one.

Our second basic building block will be \(w_0^{-2}w_2^2\) that is
pictured below.
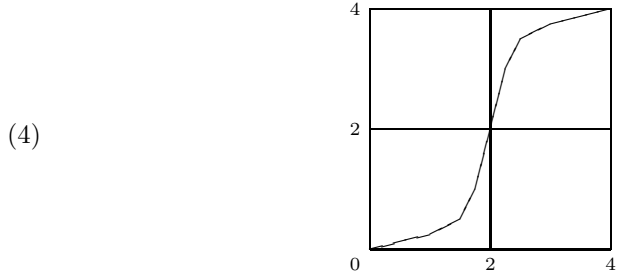
\mymargin{SecondBlock}\begin{equation}\label{SecondBlock}
\begin{split}
\xy
(0,0); (0,32)**@{-}; (32,32)**@{-}; (32,0)**@{-}; (0,0)**@{-};
(0,0); (8,2)**@{-}; (12,4)**@{-}; (14,8)**@{-}; (16,16)**@{-};
(18,24)**@{-}; (20,28)**@{-}; (24,30)**@{-}; (32,32)**@{-};
(0,16); (32,16)**@{-}; (16,0); (16,32)**@{-};
(-2,-2)*{\scriptstyle 0};
(-2,16)*{\scriptstyle 2};
(-2,32)*{\scriptstyle 4};
(16,-2)*{\scriptstyle 2};
(32,-2)*{\scriptstyle 4};
\endxy
\end{split}
\end{equation}

We consider a composition of \(m\) copies of the block from
\tref{SecondBlock} translated by muptiples of 4 as follows:
\[(w_0^{-2}w_2^2)(w_4^{-2}w_6^2) \cdots (w_{4(m-1)}^{-2} 
w_{4(m-1)+2}^2).\]
The graph looks something like the picture below, but the small
scale prevents a truly accurate picture.

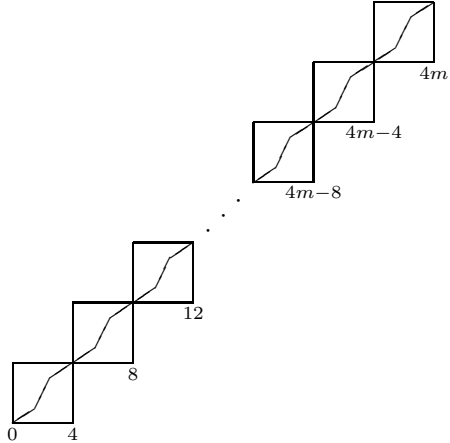
\newcommand{\DOUBLEBLOCK}
{
\xy
(0,0); (0,8)**@{-}; (8,8)**@{-}; (8,0)**@{-}; (0,0)**@{-};
(0,0); (3,2)**@{-}; (5,6)**@{-}; (8,8)**@{-};
\endxy
}

\[
\xy
(-4,-5)*{\scriptstyle 0};
(4,-5)*{\scriptstyle 4};
(0,0)*{\DOUBLEBLOCK};
(8,8)*{\DOUBLEBLOCK};
(12,3)*{\scriptstyle 8};
(16,16)*{\DOUBLEBLOCK};
(20,11)*{\scriptstyle 12};
(22,22)*{\cdot};
(24,24)*{\cdot};
(26,26)*{\cdot};
(32,32)*{\DOUBLEBLOCK};
(36,27)*{\scriptstyle 4m-8};
(40,40)*{\DOUBLEBLOCK};
(44,35)*{\scriptstyle 4m-4};
(48,48)*{\DOUBLEBLOCK};
(52,43)*{\scriptstyle 4m};
\endxy
\]

If we conjugate this function by a translation by \(b>0\) (so that
it is the product \((w_b^{-2}w_{b+2}^2)(w_{b+4}^{-2} w_{b+6}^2)
\cdots\) etc.), then we get a function with graph similar to that
above and whose support is on \([b,b+4m]\).  Call this function
\(g_0\).  To make the next discussion easier, we take \(b\) to be a
multiple of 4.

The function \(g_0\) has the property that positive powers of
\(g_0\) have the multiples of 4 in \([b,b+4m]\) as attracting fixed
points and the odd multiples of \(2\) in \([b,b+4m]\) as repelling
fixed points.

If we conjujgate \(g_0\) by translation by \(1\), then we get a
function \(g_1\) whose support is on \(J=[b+1, b+4m+1]\), which has
the elements from \(\Z\) that are equal to 1 modulo 4 in \(J\) as
attracting fixed points, and which has the elements from \(\Z\) that
are equal to 3 modulo 4 in \(J\) as repelling fixed points.

Let \(n=2m-3\).  It is standard that given any reduced word in
\(g_0\) and \(g_1\) of length less than \(n\), then the function
corresponding to that word cannot be the identity, and thus the pair
\(g_0\) and \(g_1\) cannot satisfy any relation shorter than \(n\).
One sees this as follows.  Associate to each of the four elements
\(g_0^{\pm1}\) and \(g_1^{\pm1}\) its set of attracting fixed points
in \([b,b+4m]\).  Thus \(g_0\) is associated to the integers in
\([b,b+4m]\) that are equal to 0 modulo 4, \(g_0^{-1}\) is
associated to the integers in \([b,b+4m]\) that are equal to 2
modulo 4, and so forth.  Let \(w\) be a reduced word in
\(g_i^{\pm1}\), \(i=0,1\).  Let \(\zeta\) be an integer within two
of \(b+2m\) in a set associated with neither the last letter in
\(w\), nor the inverse of the first letter.  One then shows
inductively on the length of \(w\) (based on the fact that \(w\) is
reduced), that the image of \(\zeta\) under \(w\) is within
\(\frac12\) of a point in the set associated to the last letter.
Since the image of \(\zeta\) never moves more than one under the
action of each successive letter in \(w\), and since the induction
persists as long as the image of \(\zeta\) stays within
\([b+1,b+4m]\), the induction will survive as long as the length of
\(w\) is not more than \(n\).  (We start the interval in the
previouis sentence at \(b+1\) since the behavior of \(g_1\) is under
the control of the blocks in \tref{BasicBlocks} only starting at
\(b+1\).) 

\subsection{The generators themselves}

We wish to combine the function \(g_0\) with \(x_0\).  This cannot
be done directly since their behaviors at \(b\) do not match.  At
\(b\), the first is fixed and the second translates by 1.  Thus we
will alter \(g_0\) in the interval \([b-4,b]\) so as to agree with
the picture below on the right instead of its originally defined
behavior as pictured below on the left.  
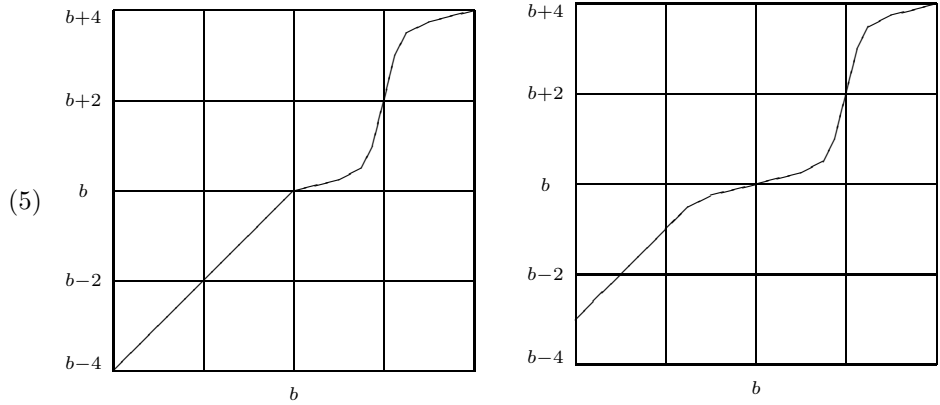
\mymargin{LeftEnd}\begin{equation}\label{LeftEnd}
\begin{split}
\xy 
(0,0); (0,48)**@{-};(48,48)**@{-}; (48,0)**@{-}; (0,0)**@{-}; 
(12,0); (12,48)**@{-};
(24,0); (24,48)**@{-};
(36,0); (36,48)**@{-};
(0,12); (48,12)**@{-}; 
(0,24); (48,24)**@{-}; 
(0,36); (48,36)**@{-}; 
(-4,1)*{\scriptstyle b-4}; 
(-4,12)*{\scriptstyle b-2}; 
(-4,24)*{\scriptstyle b}; 
(-4,36)*{\scriptstyle b+2}; 
(-4,47)*{\scriptstyle b+4};
(24,-3)*{\scriptstyle b}; 
(0,0); (24,24)**@{-}; 
(30,25.5)**@{-}; (33,27)**@{-};
(34.5,30)**@{-}; (37.5,42)**@{-}; 
(39,45)**@{-}; (42,46.5)**@{-};
(48,48)**@{-}; 
\endxy 
\qquad 
\xy 
(0,0); (0,48)**@{-};(48,48)**@{-}; (48,0)**@{-}; (0,0)**@{-}; 
(12,0); (12,48)**@{-};
(24,0); (24,48)**@{-};
(36,0); (36,48)**@{-};
(0,12); (48,12)**@{-}; 
(0,24); (48,24)**@{-}; 
(0,36); (48,36)**@{-}; 
(-4,1)*{\scriptstyle b-4}; 
(-4,12)*{\scriptstyle b-2}; 
(-4,24)*{\scriptstyle b}; 
(-4,36)*{\scriptstyle b+2}; 
(-4,47)*{\scriptstyle b+4};
(24,-3)*{\scriptstyle b}; 
(0,6); (12,18)**@{-}; (15,21)**@{-}; (18,22.5)**@{-};
(24,24)**@{-}; 
(30,25.5)**@{-}; (33,27)**@{-};
(34.5,30)**@{-}; (37.5,42)**@{-}; 
(39,45)**@{-}; (42,46.5)**@{-};
(48,48)**@{-}; 
\endxy 
\end{split}
\end{equation}
The exact
values involved in the new \(g_0\) will be made clear in Section
\ref{Details}.  What we need to know now is that the function on the
right acts as translation by 1 at  \(b-2\).

We must similarly change the behavior above \(b+4m\) so that its behavior
is as shown below. Once again, exact values will be given in Section
\ref{Details}.  For now we only need to know that the function acts as
translation by 1 on \([b+4m+3,\infty)\).
\mymargin{RightEnd}
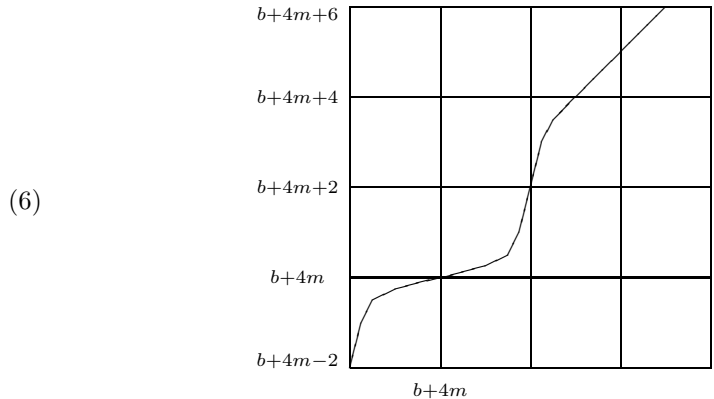
\begin{equation}\label{RightEnd}
\begin{split}
\xy
(0,0); (0,48)**@{-};(48,48)**@{-}; (48,0)**@{-}; (0,0)**@{-}; 
(12,0); (12,48)**@{-};
(24,0); (24,48)**@{-};
(36,0); (36,48)**@{-};
(0,12); (48,12)**@{-}; 
(0,24); (48,24)**@{-}; 
(0,36); (48,36)**@{-}; 
(-7,1)*{\scriptstyle b+4m-2}; 
(-7,12)*{\scriptstyle b+4m}; 
(-7,24)*{\scriptstyle b+4m+2}; 
(-7,36)*{\scriptstyle b+4m+4}; 
(-7,47)*{\scriptstyle b+4m+6};
(12,-3)*{\scriptstyle b+4m}; 
(0,0); (1.5,6)**@{-}; 
(3,9)**@{-}; (6,10.5)**@{-};
(12,12)**@{-}; 
(18,13.5)**@{-}; (21,15)**@{-};
(22.5,18)**@{-}; (24,24)**@{-};
(25.5,30)**@{-}; (27, 33)**@{-};
(30,36)**@{-};
(42,48)**@{-};
\endxy
\end{split}
\end{equation}

Now we can define \(X_0\) to agree with \(x_0\) on \([0,b-2]\), with
the modified \(g_0\) on \([b-2,b+4m+3]\) and with \(x_0\) again on
\((b+4m+3, \infty)\).

We make similar modifications to \(g_1\) and end up with an \(X_1\)
that is \(X_0\) conjugated by a translation by 1.  In particular
\(X_1\) agrees with \(x_1\) on \([0,b-1]\).  

We take \(m\) to be a positive integer.  We let \(i=m+2\).

In the following definitions, we set \(X_i=X_0^{1-i}X_1X_0^{i-1}\)
for each \(i>1\)
to parallel the relation \(x_i=x_0^{1-i}x_1x_0^{i-1}\) that holds in \(F\).
We now define:
\mymargin{TheDefs}\begin{equation}\label{TheDefs}
\begin{split}
C &= X_0^2 X_1^2 X_0^{-2} X_1^{-2}, \\
S &= X_0 X_2 X_1^{-2}, \\
T &= X_0^2 X_2 X_4 X_3^{-2} X_1^{-1} X_0^{-1}, \\
\Sigma &= C^{-i} S C^i, \\
\Theta &= C^{-i} T C^i, 
\end{split}
\qquad
\qquad
\qquad
\begin{split}
Z &= [S, \Sigma] = S \Sigma S^{-1} \Sigma^{-1}, \\
W &= [T, \Theta] = T \Theta T^{-1} \Theta^{-1}, \\
P &= Z^{-1} W, \\
Q &= X_1^{-1} P X_1 P^{-1}, \\
H &= X_1^{-2} Q X_1^2, \\
K &= X_1 H X_1^{-1}.
\end{split}
\end{equation}

As mentioned above, we define the corresponding lower case symbols, 
\(c\) for \(C\), \(\theta\) for \(\Theta\), etc., as the corresponding
words in the \(x_i\).

It is clear that \(X_0\) and \(X_1\) satisfy (a) and (b) of Lemma
\ref{GenLem}.  

In the proposition below, the elements \(y_1\) and \(w_1\) are as
defined in \tref{YZDef} and \tref{WIDef}, and \(w_1\) is not related
to the \(w\) that corresponds to \(W\) defined in \tref{TheDefs}.

\begin{prop}\mylabel{PROP}  The following hold for all sufficiently
large values of \(b\).
{\begin{enumerate}\renewcommand{\theenumi}{\roman{enumi}}
\item The symbols defined in the right hand column of \tref{TheDefs}
represent the same elements of \(F\) as the corresponding lower case
symbols. 
\item We have the equalities \(H=w_1^{-1}\) and \(K=y_1^{-1}\).  
\item The elements \(X_1\) and \(X_2\) satisfy no relation of 
length less than \(2m-3\). 
\item Hypothesis (d) of Lemma \ref{GenLem} is satisfied.
\end{enumerate}}
\end{prop}

Since (ii) in the above proposition gives (c) of Lemma \ref{GenLem},
we have that \(X_0\) and \(X_1\) generate \(F\).  This and (iii) of
the proposition give Theorem \ref{THM}.

We add a bit more to the outline before surrendering this paper to
the details.  Items (i) and (ii) are the most technical.  

For (i), we will divide \(\R_{\ge0}\) into two regions, \([0,b-5)\)
and \((b-5,\infty)\).  With \(b\) sufficiently large, the lower case
symbols corresponding to  the words defined in \tref{TheDefs} will
have supports in \([0,b-5)\).  This is the only requirement on
\(b\) and contains the meaning of ``sufficiently large.''

We will show that \(C\) has two parts to its support.  There will be
a part in \([0,b-5)\) where \(C\) and \(c\) agree and a part in a
closed interval \(I_C\) in \((b-5,\infty)\) for which \(\zeta
C>\zeta\) for all \(\zeta\) in the interior of \(I_C\).

The functions defined as \(S\) and \(T\) will also have two parts to
their support.  There will be a part in \([0,b-5)\) where \(S\) and
\(T\) will agree with \(s\) and \(t\), respectively, and a part with
closure in the interior of \(I_C\).

We will show that the the conjugates of \(S\) and \(T\) by \(C^i\)
will have the parts of their supports in \((b-5,\infty)\) disjoint
from the supports of \(S\) and \(T\).  Thus the commutators \(Z\)
and \(W\) will be trivial on \((b-5,\infty)\).  This will verify (i)
for \(Z\) and \(W\), and the truth of (i) for the rest will follow
easily.

The truth of (ii) will follow from a long algebraic calculation.

The argument for (iii) has already been given for the original
functions \(g_0\) and \(g_1\).  The argument applies to \(X_0\) and
\(X_1\) since these agree with the original \(g_0\) and \(g_1\),
respectively, on the interval \([b+1,b+4m]\) needed for the argument.

The truth of (iv) will follow from the information gathered in the
arguments for (ii) and from the definitions of \(X_0\) and \(X_1\).

The next section gives the details needed to verify the truth of
(i), (ii) and (iv) of Proposition \ref{PROP}.  This will complete
the proof of Theorem \ref{THM}.

\section{Details}\mylabel{Details}

The elements in \tref{TheDefs} were found using a computer program
for doing calculations in \(F\) that was written by the author
almost 20 years ago.  It was written partly as an exercise in
learing how to use a compiler compiler (sic).  The elements in
\tref{TheDefs} were found by ``experiment guided by experience.''
Once elements were found with the required properties, the problem
of how to write a proof of Therem \ref{THM} arose.  Listing the
definitions in \tref{TheDefs} and then instructing the reader to the
use the program to check the claimed behaviors was not an option
since the program is very large and its inner workings are (after
almost 20 years) opaque even to the author.  It was then discovered
that the algebra behind (ii) of Proposition \ref{PROP} was not that
bad, and that the ``rectangular diagrams'' of Thurston as described
in \cite{CFP} made the verification of all that is needed for (i) of
Proposition \ref{PROP} very visual.  The result, while manageable,
is still not attractive.

A failed attempt was made to find more attractive examples.  There
may very well be less complicated generators, or ones whose
properties are easier to prove.  However if such examples exist,
they seem hard to find.

\subsection{Proposition \ref{PROP} part (ii)}

We start with the more algebraic calculations.  We will not end with
a proof of (ii), but a proof that (ii) follows from (i).  If (i)
holds, then \(H=h\) and \(K=k\) and proving that \(h=w_1^{-1}\) and
\(k=y_1^{-1}\) will give (ii).  Thus we analyze the lower case symbols.

We put the lower case symbols in normal form.  Some, such as \(s\)
and \(t\) are already given in normal form.

First \[c=x_0^2 (x_1^2 x_3^{-2}) x_0^{-2}.\]

Next \[ \begin{split} c^i &= x_0^2 (x_1^2 x_3^{-2})^i x_0^{-2} \\
  &= x_0^2 (x_1^2 x_3^{-2} x_1^2 x_3^{-2} \cdots x_1^2 x_3^{-2}) x_0^{-2} \\
  &= x_0^2 
     (x_1^{2i} x_{2i+1}^{-2} x_{2i-1}^{-2} \cdots x_5^{-2} x_3^{-2})
     x_0^{-2}. \end{split}\]

We skip \(\sigma\) and \(\theta\) since they are absorbed into
\(z\) and \(w\).  We start with \(z\).

We have \[\begin{split} z = &s (c^{-i} s c^i) s^{-1} (c^{-i} s^{-1} c^i) \\
   =& (x_0x_2x_1^{-2})
      (x_0^2 
       x_3^{2} x_5^{2} \cdots x_{2i-1}^{2} x_{2i+1}^{2} x_1^{-2i}
        x_0^{-2}) \\
    & (x_0x_2x_1^{-2})
      (x_0^2 
       x_1^{2i} x_{2i+1}^{-2} x_{2i-1}^{-2} \cdots x_5^{-2} x_3^{-2}
        x_0^{-2}) \\
     & (x_1^2 x_2^{-1} x_0^{-1})
      (x_0^2 
       x_3^{2} x_5^{2} \cdots x_{2i-1}^{2} x_{2i+1}^{2} x_1^{-2i}
        x_0^{-2}) \\
      & (x_1^2 x_2^{-1} x_0^{-1})
       (x_0^2 
       x_1^{2i} x_{2i+1}^{-2} x_{2i-1}^{-2} \cdots x_5^{-2} x_3^{-2}
        x_0^{-2}).
       \end{split}\]

Now we move the appearances of \(x_0^{\pm2}\) from the long parenthesized
expressions on the right of each line.
\[\begin{split}
   z=& (x_0^3 x_4 x_3^{-2})
      ( 
       x_3^{2} x_5^{2} \cdots x_{2i-1}^{2} x_{2i+1}^{2} x_1^{-2i}
        ) \\
    & (x_0 x_4 x_3^{-2})
      (
       x_1^{2i} x_{2i+1}^{-2} x_{2i-1}^{-2} \cdots x_5^{-2} x_3^{-2}
        ) \\
     & (x_3^2 x_4^{-1} x_0^{-1})
      (
       x_3^{2} x_5^{2} \cdots x_{2i-1}^{2} x_{2i+1}^{2} x_1^{-2i}
        ) \\
      & (x_3^2 x_4^{-1} x_0^{-1})
       (
       x_1^{2i} x_{2i+1}^{-2} x_{2i-1}^{-2} \cdots x_5^{-2} x_3^{-2}
        x_0^{-2}).
\end{split}\]

We cancel adjacent inverse items.
\[\begin{split}
   z=& (x_0^3 x_4 )
      ( 
       x_5^{2} \cdots x_{2i-1}^{2} x_{2i+1}^{2} x_1^{-2i}
        ) \\
    & (x_0 x_4 x_3^{-2})
      (
       x_1^{2i} x_{2i+1}^{-2} x_{2i-1}^{-2} \cdots x_5^{-2}
        ) \\
     & (x_4^{-1} x_0^{-1})
      (
       x_3^{2} x_5^{2} \cdots x_{2i-1}^{2} x_{2i+1}^{2} x_1^{-2i}
        ) \\
      & (x_3^2 x_4^{-1} x_0^{-1})
       (
       x_1^{2i} x_{2i+1}^{-2} x_{2i-1}^{-2} \cdots x_5^{-2} x_3^{-2}
        x_0^{-2}).
\end{split}\]

Now we move the 
remaining  appearances of \(x_0^{\pm1}\) from the middle of the expression.
\[\begin{split}
   z=& (x_0^4 x_5 )
      ( 
       x_6^{2} \cdots x_{2i}^{2} x_{2i+2}^{2} x_2^{-2i}
        ) \\
    & (x_4 x_3^{-2})
      (
       x_1^{2i} x_{2i+1}^{-2} x_{2i-1}^{-2} \cdots x_5^{-2}
        ) \\
     & (x_4^{-1} )
      (
       x_4^{2} x_6^{2} \cdots x_{2i}^{2} x_{2i+2}^{2} x_2^{-2i}
        ) \\
      & (x_4^2 x_5^{-1} )
       (
       x_3^{2i} x_{2i+3}^{-2} x_{2i+1}^{-2} \cdots x_7^{-2} x_5^{-2}
        x_0^{-4}).
\end{split}\]

We cancel one adjacent pair.
\[\begin{split}
   z=& (x_0^4 x_5 )
      ( 
       x_6^{2} \cdots x_{2i}^{2} x_{2i+2}^{2} x_2^{-2i}
        ) \\
    & (x_4 x_3^{-2})
      (
       x_1^{2i} x_{2i+1}^{-2} x_{2i-1}^{-2} \cdots x_5^{-2}
        ) \\
     & (
       x_4 x_6^{2} \cdots x_{2i}^{2} x_{2i+2}^{2} x_2^{-2i}
        ) \\
      & (x_4^2 x_5^{-1} )
       (
       x_3^{2i} x_{2i+3}^{-2} x_{2i+1}^{-2} \cdots x_7^{-2} x_5^{-2}
        x_0^{-4}).
\end{split}\]

We move the \(x_1^{2i}\) from the second line and the \(x_2^{-2i}\) from
the third line.
\[\begin{split}
   z=& (x_0^4 x_1^{2i} x_{2i+5} )
      ( 
       x_{2i+6}^{2} \cdots x_{4i}^{2} x_{4i+2}^{2} x_{2i+2}^{-2i}
        ) \\
    & (x_{2i+4} x_{2i+3}^{-2})
      (
       x_{2i+1}^{-2} x_{2i-1}^{-2} \cdots x_5^{-2}
        ) \\
     & (
       x_4 x_6^{2} \cdots x_{2i}^{2} x_{2i+2}^{2}
        ) \\
      & (x_{2i+4}^2 x_{2i+5}^{-1} )
       (
       x_{2i+3}^{2i} x_{4i+3}^{-2} x_{4i+1}^{-2} \cdots 
                  x_{2i+7}^{-2} x_{2i+5}^{-2}
        x_2^{-2i} x_0^{-4}).
\end{split}\]

We move the \(x_4\) from the beginning of the third line.
\[\begin{split}
   z=& (x_0^4 x_1^{2i} x_4 x_{2i+6} )
      ( 
       x_{2i+7}^{2} \cdots x_{4i+1}^{2} x_{4i+3}^{2} x_{2i+3}^{-2i}
        ) \\
    & (x_{2i+5} x_{2i+4}^{-2})
      (
       x_{2i+2}^{-2} x_{2i}^{-2} \cdots x_6^{-2}
        ) \\
     & (
       x_6^{2} \cdots x_{2i}^{2} x_{2i+2}^{2}
        ) \\
      & (x_{2i+4}^2 x_{2i+5}^{-1} )
       (
       x_{2i+3}^{2i} x_{4i+3}^{-2} x_{4i+1}^{-2} \cdots 
                  x_{2i+7}^{-2} x_{2i+5}^{-2}
        x_2^{-2i} x_0^{-4}).
\end{split}\]

We cancel inverse pairs.
\[
   z= (x_0^4 x_1^{2i} x_4 x_{2i+6} )
       (x_{2i+5}^{-2} x_2^{-2i} x_0^{-4}).
\]

Now we work on \(w\).
\[\begin{split} w = &t (c^{-i} t c^i) t^{-1} (c^{-i} t^{-1} c^i) \\
   =& (x_0^2 x_2 x_4 x_3^{-2} x_1^{-1} x_0^{-1})
      (x_0^2 
       x_3^{2} x_5^{2} \cdots x_{2i-1}^{2} x_{2i+1}^{2} x_1^{-2i}
        x_0^{-2}) \\
    & (x_0^2 x_2 x_4 x_3^{-2} x_1^{-1} x_0^{-1})
      (x_0^2 
       x_1^{2i} x_{2i+1}^{-2} x_{2i-1}^{-2} \cdots x_5^{-2} x_3^{-2}
        x_0^{-2}) \\
     & (x_0 x_1 x_3^2 x_4^{-1} x_2^{-1} x_0^{-2})
      (x_0^2 
       x_3^{2} x_5^{2} \cdots x_{2i-1}^{2} x_{2i+1}^{2} x_1^{-2i}
        x_0^{-2}) \\
      & (x_0 x_1 x_3^2 x_4^{-1} x_2^{-1} x_0^{-2})
       (x_0^2 
       x_1^{2i} x_{2i+1}^{-2} x_{2i-1}^{-2} \cdots x_5^{-2} x_3^{-2}
        x_0^{-2}) \\
   =& (x_0^2 x_2 x_4 x_3^{-2} x_1^{-1} )
      (x_0
       x_3^{2} x_5^{2} \cdots x_{2i-1}^{2} x_{2i+1}^{2} x_1^{-2i}
        ) \\
    & (x_2 x_4 x_3^{-2} x_1^{-1} )
      (x_0
       x_1^{2i} x_{2i+1}^{-2} x_{2i-1}^{-2} \cdots x_5^{-2} x_3^{-2}
        x_0^{-1}) \\
     & (x_1 x_3^2 x_4^{-1} x_2^{-1} )
      (
       x_3^{2} x_5^{2} \cdots x_{2i-1}^{2} x_{2i+1}^{2} x_1^{-2i}
        x_0^{-1}) \\
      & (x_1 x_3^2 x_4^{-1} x_2^{-1} )
       (
       x_1^{2i} x_{2i+1}^{-2} x_{2i-1}^{-2} \cdots x_5^{-2} x_3^{-2}
        x_0^{-2}).
       \end{split}\]

Moving the internal \(x_0^{\pm1}\) gives the following.
\[\begin{split} w 
     =& (x_0^4 x_4 x_6 x_5^{-2} x_3^{-1} )
      (
       x_4^{2} x_6^{2} \cdots x_{2i}^{2} x_{2i+2}^{2} x_2^{-2i}
        ) \\
    & (x_3 x_5 x_4^{-2} x_2^{-1} )
      (
       x_1^{2i} x_{2i+1}^{-2} x_{2i-1}^{-2} \cdots x_5^{-2} x_3^{-2}
        ) \\
     & (x_2 x_4^2 x_5^{-1} x_3^{-1} )
      (
       x_4^{2} x_6^{2} \cdots x_{2i}^{2} x_{2i+2}^{2} x_2^{-2i}
        ) \\
      & (x_3 x_5^2 x_6^{-1} x_4^{-1} )
       (
       x_3^{2i} x_{2i+3}^{-2} x_{2i+1}^{-2} \cdots x_7^{-2} x_5^{-2}
        x_0^{-4}).
       \end{split}\]

We move the \(x_1^{2i}\) from the second line and the \(x_2^{-2i}\) from
the third line.
\[\begin{split} w 
     =& (x_0^4 x_1^{2i} x_{2i+4} x_{2i+6} x_{2i+5}^{-2} x_{2i+3}^{-1} )
      (
       x_{2i+4}^{2} x_{2i+6}^{2} \cdots 
        x_{4i}^{2} x_{4i+2}^{2} x_{2i+2}^{-2i}
        ) \\
    & (x_{2i+3} x_{2i+5} x_{2i+4}^{-2} x_{2i+2}^{-1} )
      (
       x_{2i+1}^{-2} x_{2i-1}^{-2} \cdots x_5^{-2} x_3^{-2}
        ) \\
     & (x_2 x_4^2 x_5^{-1} x_3^{-1} )
      (
       x_4^{2} x_6^{2} \cdots x_{2i}^{2} x_{2i+2}^{2}
        ) \\
      & (x_{2i+3} x_{2i+5}^2 x_{2i+6}^{-1} x_{2i+4}^{-1} )
       (
       x_{2i+3}^{2i} x_{4i+3}^{-2} x_{4i+1}^{-2} \cdots 
        x_{2i+7}^{-2} x_{2i+5}^{-2}
        x_2^{-2i} x_0^{-4}).
       \end{split}\]

We move \(x_2\) and \(x_3^{-1}\) from the third line and cancel
some adjacent pairs.
\[\begin{split} w 
     =& (x_0^4 x_1^{2i} x_2
         x_{2i+5} x_{2i+7} x_{2i+6}^{-2} x_{2i+4}^{-1} )
      (
       x_{2i+5}^{2} x_{2i+7}^{2} \cdots 
        x_{4i+1}^{2} x_{4i+3}^{2} x_{2i+3}^{-2i}
        ) \\
    & (x_{2i+4} x_{2i+6} x_{2i+5}^{-2} x_{2i+3}^{-1} )
      (
       x_{2i+2}^{-2} x_{2i}^{-2} \cdots x_6^{-2}
        ) \\
     & 
      (
       x_5 x_7^{2} \cdots x_{2i+1}^{2} x_{2i+3}^{2}
        ) \\
      & (x_{2i+4} x_{2i+6}^2 x_{2i+7}^{-1} x_{2i+5}^{-1} )
       (
       x_{2i+4}^{2i} x_{4i+4}^{-2} x_{4i+2}^{-2} \cdots 
        x_{2i+8}^{-2} x_{2i+6}^{-2}
       x_3^{-1} x_2^{-2i} x_0^{-4}).
       \end{split}\]

We move the \(x_4\) from the beginning of the third line.
\[\begin{split} w 
     =& (x_0^4 x_1^{2i} x_2 x_5
         x_{2i+6} x_{2i+8} x_{2i+7}^{-2} x_{2i+5}^{-1} )
      (
       x_{2i+6}^{2} x_{2i+8}^{2} \cdots 
        x_{4i+2}^{2} x_{4i+4}^{2} x_{2i+4}^{-2i}
        ) \\
    & (x_{2i+5} x_{2i+7} x_{2i+6}^{-2} x_{2i+4}^{-1} )
      (
       x_{2i+3}^{-2} x_{2i+1}^{-2} \cdots x_7^{-2}
        ) \\
     & 
      (
       x_7^{2} \cdots x_{2i+1}^{2} x_{2i+3}^{2}
        ) \\
      & (x_{2i+4} x_{2i+6}^2 x_{2i+7}^{-1} x_{2i+5}^{-1} )
       (
       x_{2i+4}^{2i} x_{4i+4}^{-2} x_{4i+2}^{-2} \cdots 
        x_{2i+8}^{-2} x_{2i+6}^{-2}
       x_3^{-1} x_2^{-2i} x_0^{-4}).
       \end{split}\]

We cancel inverse pairs.
\[ w 
     = (x_0^4 x_1^{2i} x_2 x_5
         x_{2i+6} x_{2i+8} x_{2i+7}^{-2} x_{2i+5}^{-1} )
       (x_3^{-1} x_2^{-2i} x_0^{-4}).
\]

We continue with \(p\), \(q\), \(h\), and \(k\).
\[\begin{split} 
p = &z^{-1}w \\
  = &(x_0^4 x_1^{2i} x_4 x_{2i+6} 
       x_{2i+5}^{-2} x_2^{-2i} x_0^{-4})^{-1} \\
    &(x_0^4 x_1^{2i} x_2 x_5
         x_{2i+6} x_{2i+8} x_{2i+7}^{-2} x_{2i+5}^{-1} )
       (x_3^{-1} x_2^{-2i} x_0^{-4}) \\
  = &(x_0^4 x_2^{2i} x_{2i+5}^2
       x_{2i+6}^{-1} x_4^{-1} x_1^{-2i} x_0^{-4}) \\
    &(x_0^4 x_1^{2i} x_2 x_5
         x_{2i+6} x_{2i+8} x_{2i+7}^{-2} x_{2i+5}^{-1} 
       x_3^{-1} x_2^{-2i} x_0^{-4}) \\
  = &(x_0^4 x_2^{2i} x_{2i+5}^2
       x_{2i+6}^{-1} x_4^{-1} )
    (x_2 x_5
         x_{2i+6} x_{2i+8} x_{2i+7}^{-2} x_{2i+5}^{-1} 
       x_3^{-1} x_2^{-2i} x_0^{-4}) \\
  = &(x_0^4 x_2^{2i+1} x_{2i+6}^2
       x_{2i+7}^{-1} x_5^{-1} ) 
    (x_5
         x_{2i+6} x_{2i+8} x_{2i+7}^{-2} x_{2i+5}^{-1} 
       x_3^{-1} x_2^{-2i} x_0^{-4}) \\
  = &(x_0^4 x_2^{2i+1} x_{2i+6}^2
       x_{2i+7}^{-1} ) 
    (
         x_{2i+6} x_{2i+8} x_{2i+7}^{-2} x_{2i+5}^{-1} 
       x_3^{-1} x_2^{-2i} x_0^{-4}) \\
  = &(x_0^4 x_2^{2i+1} x_{2i+6}^3
         x_{2i+7}^{-2} x_{2i+5}^{-1} 
       x_3^{-1} x_2^{-2i} x_0^{-4}) \\
  = &(x_0^3 x_1^{2i+1} x_{2i+5}^3
         x_{2i+6}^{-2} x_{2i+4}^{-1} 
       x_2^{-1} x_1^{-2i} x_0^{-3})
\end{split}\]

Now.
\[\begin{split}
x_1^{-1}px_1 
  =& x_1^{-1} (x_0^3 x_1^{2i+1} x_{2i+5}^3
         x_{2i+6}^{-2} x_{2i+4}^{-1} 
       x_2^{-1} x_1^{-2i} x_0^{-3}) x_1 \\
  =&  x_0^3 x_4^{-1} x_1^{2i+1} x_{2i+5}^3
         x_{2i+6}^{-2} x_{2i+4}^{-1} 
       x_2^{-1} x_1^{-2i} x_4 x_0^{-3}  \\
  =&  x_0^3 x_1^{2i+1} x_{2i+5}^{-1} x_{2i+5}^3
         x_{2i+6}^{-2} x_{2i+4}^{-1} x_{2i+5}
       x_2^{-1} x_1^{-2i} x_0^{-3}  \\
  =&  x_0^3 x_1^{2i+1} x_{2i+5}^2
         x_{2i+6}^{-2} x_{2i+4}^{-1} x_{2i+5}
       x_2^{-1} x_1^{-2i} x_0^{-3}  \\
  =&  x_0^3 x_1^{2i+1} x_{2i+5}^2
         x_{2i+6}^{-2} x_{2i+6} x_{2i+4}^{-1}
       x_2^{-1} x_1^{-2i} x_0^{-3}  \\
  =&  x_0^3 x_1^{2i+1} x_{2i+5}^2
         x_{2i+6}^{-1} x_{2i+4}^{-1}
       x_2^{-1} x_1^{-2i} x_0^{-3}  .
\end{split}\]

So.
\[\begin{split}
q = &x_1^{-1} p x_1 p^{-1} \\
    = &(x_0^3 x_1^{2i+1} x_{2i+5}^2
         x_{2i+6}^{-1} x_{2i+4}^{-1}
       x_2^{-1} x_1^{-2i} x_0^{-3} ) \\
      &(x_0^3 x_1^{2i+1} x_{2i+5}^3
         x_{2i+6}^{-2} x_{2i+4}^{-1} 
       x_2^{-1} x_1^{-2i} x_0^{-3})^{-1} \\
     = &(x_0^3 x_1^{2i+1} x_{2i+5}^2
         x_{2i+6}^{-1} x_{2i+4}^{-1}
       x_2^{-1} x_1^{-2i} x_0^{-3} ) \\
      &(x_0^3 x_1^{2i} x_2 x_{2i+4}
         x_{2i+6}^{2} x_{2i+5}^{-3} 
       x_1^{-(2i+1)} x_0^{-3}) \\
     = &x_0^3 x_1^{2i+1} x_{2i+5}^2
       x_{2i+6} x_{2i+5}^{-3} 
       x_1^{-(2i+1)} x_0^{-3} \\
     = &x_0^3 x_{4}^2
       x_{5} x_{4}^{-3} 
       x_0^{-3} \\
     = &x_{1}^2
       x_{2} x_{1}^{-3}.
\end{split}\]

Finally.
\[\begin{split}
  h &= x_1^{-2} q x_1^2 = x_1^{-2} (x_{1}^2
       x_{2} x_{1}^{-3}) x_1^2 = x_2 x_1^{-1} = w_1^{-1}, \\
  k &= x_1 h x_1^{-1} = x_1 x_2 x_1^{-2} = y_1^{-1}.
\end{split}\]

This completes the proof of (ii) from (i).  The calculations above
might give the impression that generating \(F\) (or more to the
point, generating some \(F_{[a,b]}\)) is a rare phenomenon.  This is
not so.  While certain obvious obstructions make the generation of
some \(F_{[a,b]}\) a ``probability zero'' event, it is still
surprisingly easy to arrange that it happens.

\subsection{Proposition \ref{PROP} Part (i)}

Almost all of the effort here will go to understanding the supports
of the elements in \tref{TheDefs}.  What we do here will also supply
the missing details about the defined behavior of \(g_0\) and
\(g_1\) that were only vaguely described in \tref{LeftEnd} and
\tref{RightEnd}.

\subsubsection{First look at the supports}\mylabel{FirstLookSec}

We start with some preliminary estimates that relate to \(b\).
These will be sharpened later.

Each of the elements defined in \tref{TheDefs} acts on \(\R_{\ge0}\)
as the identity on certain intervals.  It will be necessary to know
something about what these intervals are.  We will look at both the
upper and lower case symbols.

We note that the symbols in \tref{TheDefs} occur in three groups.
The symbols \(C\), \(S\) and \(T\) are defined in terms of the
\(X_i\), the symbols \(\Sigma\), \(\Theta\), \(Z\), \(W\) and \(P\)
are defined in terms of \(C\), \(S\) and \(T\), and the last three
symbols are defined in terms of \(P\) and \(X_1\).

We consider \(C\), \(S\) and \(T\) first.  When written in terms of
\(X_0\) and \(X_1\) using \(X_i=X_0^{1-i}X_1X_0^{i-1}\), we see that
\[ \begin{split} T &= X_0^2X_2X_4X_3^{-2} X_1^{-1}X_0^{-1} \\ &= X_0
X_1 X_0^{-1} X_0^{-1} X_1 X_0 X_1^{-1} X_1^{-1} X_0 X_0 X_1^{-1}
X_0^{-1}.\end{split}\] is the longest at 12 letters.  We also note
that the total exponent sum over all the generators in each of
\(C\), \(S\) and \(T\) is zero.  Thus the sum of the positive
exponents is never more than 6.  Since \(X_0^{\pm1}\) and
\(X_1^{\pm1}\) are translations by \(\pm1\) on at least \([3,b-2]\),
it follows that if the length of \([3,b-2]\) is 12 or more, then
each of \(C\), \(S\) and \(T\) has a fixed point at 9.  From this
point we can take \(b\) to be at least 17.  We will see later that
this is overly cautious.

With \(b\) as above, it follows similarly from the fact that each of
\(x_0^{\pm1}\) and \(x_1^{\pm1}\) is translation by \(\pm1\) on at
least \([3,\infty)\) that each of \(c\), \(s\) and \(t\) has support
in \([0,9]\), and that each of \(C\), \(S\) and \(T\) agrees with
\(c\), \(s\) and \(t\), respectively, on \([0,9]\).

It now follows that each of the symbols in the second group
\(\Sigma\), \(\Theta\) \(Z\), \(W\), and \(P\) has a fixed point at
9 and that it agrees with the corresponding lower case symbol on
\([0,9]\).

Discussion of the last group can wait until it is shown that \(W=w\)
and \(Z=z\).

%
%

\newcommand{\LSLANT} 
{
\xy
(0,0); (16,0)**@{-}; (16,8)**@{-}; (0,8)**@{-}; (0,0)**@{-};
(2,0); (8,8)**@{-};
(4,0); (12,8)**@{-};
(8,0); (14,8)**@{-};
\endxy
}

\newcommand{\RSLANT} 
{
\xy
(0,0); (16,0)**@{-}; (16,8)**@{-}; (0,8)**@{-}; (0,0)**@{-};
(8,0); (2,8)**@{-};
(12,0); (4,8)**@{-};
(14,0); (8,8)**@{-};
\endxy
}

\newcommand{\RTERMP} 
{
\xy
(0,0); (32,0)**@{-}; (32,8); (0,8)**@{-}; (0,0)**@{-};
(8,0); (2,8)**@{-};
(12,0); (4,8)**@{-}; (16,0); (8,8)**@{-};
(24,0); (16,8)**@{-};
(32,0); (24,8)**@{-};
\endxy
}

\newcommand{\RTERMN} 
{
\xy
(0,0); (32,0)**@{-}; (32,8); (0,8)**@{-}; (0,0)**@{-};
(8,8); (2,0)**@{-};
(12,8); (4,0)**@{-}; (16,8); (8,0)**@{-};
(24,8); (16,0)**@{-};
(32,8); (24,0)**@{-};
\endxy
}

\newcommand{\LTERMP} 
{
\xy
(0,0); (32,0)**@{-}; (32,8)**@{-}; (0,8)**@{-};
(24,8); (30,0)**@{-};
(20,8); (28,0)**@{-}; (16,8); (24,0)**@{-};
(8,8); (16,0)**@{-};
(0,8); (8,0)**@{-};
\endxy
}

\newcommand{\LTERMN} 
{
\xy
(0,0); (32,0)**@{-}; (32,8)**@{-}; (0,8)**@{-};
(24,0); (30,8)**@{-};
(20,0); (28,8)**@{-}; (16,0); (24,8)**@{-};
(8,0); (16,8)**@{-};
(0,0); (8,8)**@{-};
\endxy
}

\newcommand{\TERMEXTP} 
{
\xy
(0,0); (8,0)**@{-}; (8,8); (0,8)**@{-};
(8,0); (0,8)**@{-};
\endxy
}

\newcommand{\TERMEXTN} 
{
\xy
(0,0); (8,0)**@{-}; (8,8); (0,8)**@{-};
(8,8); (0,0)**@{-};
\endxy
}

\newcommand{\EIGHTBOX} 
{
\xy
(0,8); (8,8)**@{-}; (8,0); (0,0)**@{-};
(4,4)*{\cdots};
\endxy
}

\newcommand{\XZEROP}     
{
\xy
(0,0)*{\RSLANT};  
(16,0)*{\LSLANT};  
(32,0)*{\RSLANT};  
(48,0)*{\LSLANT};  
(72,0)*{\RTERMP};  
(92,0)*{\TERMEXTP};  
(100,0)*{\TERMEXTP};  
\endxy
}

\newcommand{\XZERON}    
{
\xy
(0,0)*{\LSLANT};  
(16,0)*{\RSLANT};  
(32,0)*{\LSLANT};  
(48,0)*{\RSLANT};  
(72,0)*{\RTERMN};  
(92,0)*{\TERMEXTN};  
(100,0)*{\TERMEXTN};  
\endxy
}

\newcommand{\XONEP}        
{
\xy
(0,0)*{\EIGHTBOX}; 
(12,0)*{\RSLANT}; 
(28,0)*{\LSLANT}; 
(44,0)*{\RSLANT}; 
(60,0)*{\LSLANT}; 
(84,0)*{\RTERMP}; 
(104,0)*{\TERMEXTP}; 
\endxy
}

\newcommand{\XONEN}    
{
\xy
(0,0)*{\EIGHTBOX}; 
(12,0)*{\LSLANT}; 
(28,0)*{\RSLANT}; 
(44,0)*{\LSLANT}; 
(60,0)*{\RSLANT}; 
(84,0)*{\RTERMN}; 
(104,0)*{\TERMEXTN}; 
\endxy  
}

\newcommand{\RXCOORDS}
{
\xy
(0,0)*{}; (112,0)*{}; (112,4)*{}; (0,4)*{};
(-56,2)*{\scriptstyle \xi-8};
(-40,2)*{\scriptstyle \xi-6};
(-24,2)*{\scriptstyle \xi-4};
(-8,2)*{\scriptstyle \xi-2};
(8,2)*{\scriptstyle \xi};
(16,2)*{\scriptstyle \xi+1};
(24,2)*{\scriptstyle \xi+2};
(32,2)*{\scriptstyle \xi+3};
\endxy
}

\newcommand{\LXCOORDS}
{
\xy
(0,0)*{}; (112,0)*{}; (112,4)*{}; (0,4)*{};
(-56,2)*{\scriptstyle b-5};
(-40,2)*{\scriptstyle b-3};
(-24,2)*{\scriptstyle b-1};
(-16,2)*{b};
(0,2)*{\scriptstyle b+2};
(16,2)*{\scriptstyle b+4};
(32,2)*{\scriptstyle b+6};
(48,2)*{\scriptstyle b+8};
\endxy
}

\newcommand{\BIGLSLANT} 
{
\xy
(0,0); (48,0)**@{-}; (48,24)**@{-}; (0,24)**@{-}; (0,0)**@{-};
(6,0); (24,24)**@{-};
(12,0); (36,24)**@{-};
(24,0); (42,24)**@{-};
(0,-2)*{\scriptstyle 0};  
(24,-2)*{\scriptstyle 1}; 
(48,-2)*{\scriptstyle 2};
(12,-2)*{\scriptstyle \frac12};
(6,-2)*{\scriptstyle \frac14};
(0,26)*{\scriptstyle 0};  
(24,26)*{\scriptstyle 1}; 
(48,26)*{\scriptstyle 2};
(36,26)*{\scriptstyle \frac32};
(42,26)*{\scriptstyle \frac74};
(7.5,12)*{\frac14};
(19.5,12)*{\frac12};
(28.5,12)*{2};
(40.5,12)*{4};
\endxy
}

\newcommand{\BIGRSLANT} 
{
\xy
(0,0); (48,0)**@{-}; (48,24)**@{-}; (0,24)**@{-}; (0,0)**@{-};
(24,0); (6,24)**@{-};
(36,0); (12,24)**@{-};
(42,0); (24,24)**@{-};
(0,26)*{\scriptstyle 0};  
(24,26)*{\scriptstyle 1}; 
(48,26)*{\scriptstyle 2};
(12,26)*{\scriptstyle \frac12};
(6,26)*{\scriptstyle \frac14};
(0,-2)*{\scriptstyle 0};  
(24,-2)*{\scriptstyle 1}; 
(48,-2)*{\scriptstyle 2};
(36,-2)*{\scriptstyle \frac32};
(42,-2)*{\scriptstyle \frac74};
(7.5,12)*{4};
(19.5,12)*{2};
(28.5,12)*{\frac12};
(40.5,12)*{\frac14};
\endxy
}

\newcommand{\BIGRTERMP} 
{
\xy
(0,0); (48,0)**@{-}; (48,24); (0,24)**@{-}; (0,0)**@{-};
(24,0); (6,24)**@{-};
(36,0); (12,24)**@{-}; (48,0); (24,24)**@{-};
(0,26)*{\scriptstyle 0};
(6,26)*{\scriptstyle \frac14};
(12,26)*{\scriptstyle \frac12};
(24,26)*{\scriptstyle 1};
(48,26)*{\scriptstyle 2};
(0,-2)*{\scriptstyle 0};
(24,-2)*{\scriptstyle 1};
(36,-2)*{\scriptstyle \frac32};
(48,-2)*{\scriptstyle 2};
(7.5,12)*{4};
(19.5,12)*{2};
(30,12)*{1};
(42,12)*{1};
\endxy
}

\newcommand{\BIGRTERMN} 
{
\xy
(0,0); (48,0)**@{-}; (48,24); (0,24)**@{-}; (0,0)**@{-};
(24,24); (6,0)**@{-};
(36,24); (12,0)**@{-}; (48,24); (24,0)**@{-};
(0,-2)*{\scriptstyle 0};
(6,-2)*{\scriptstyle \frac14};
(12,-2)*{\scriptstyle \frac12};
(24,-2)*{\scriptstyle 1};
(48,-2)*{\scriptstyle 2};
(0,26)*{\scriptstyle 0};
(24,26)*{\scriptstyle 1};
(36,26)*{\scriptstyle \frac32};
(48,26)*{\scriptstyle 2};
(7.5,12)*{\frac14};
(19.5,12)*{\frac12};
(30,12)*{1};
(42,12)*{1};
\endxy
}

\newcommand{\BIGLTERMP} 
{
\xy
(0,0); (48,0)**@{-}; (48,24)**@{-}; (0,24)**@{-};
(0,24); (24,0)**@{-};
(12,24); (36,0)**@{-};
(24,24); (42,0)**@{-};
(0,-2)*{\scriptstyle 0};
(24,-2)*{\scriptstyle 1};
(36,-2)*{\scriptstyle \frac32};
(42,-2)*{\scriptstyle \frac74};
(48,-2)*{\scriptstyle 2};
(0,26)*{\scriptstyle 0};
(12,26)*{\scriptstyle \frac12};
(24,26)*{\scriptstyle 1};
(48,26)*{\scriptstyle 2};
(6,12)*{1};
(18,12)*{1};
(28.5,12)*{\frac12};
(40.5,12)*{\frac14};
\endxy
}

\newcommand{\BIGLTERMN} 
{
\xy
(0,0); (48,0)**@{-}; (48,24)**@{-}; (0,24)**@{-};
(0,0); (24,24)**@{-};
(12,0); (36,24)**@{-};
(24,0); (42,24)**@{-};
(0,26)*{\scriptstyle 0};
(24,26)*{\scriptstyle 1};
(36,26)*{\scriptstyle \frac32};
(42,26)*{\scriptstyle \frac74};
(48,26)*{\scriptstyle 2};
(0,-2)*{\scriptstyle 0};
(12,-2)*{\scriptstyle \frac12};
(24,-2)*{\scriptstyle 1};
(48,-2)*{\scriptstyle 2};
(6,12)*{1};
(18,12)*{1};
(28.5,12)*{2};
(40.5,12)*{4};
\endxy
}

\newcommand{\LXZEROP}
{
\xy
(-20,0)*{\TERMEXTP};  
(0,0)*{\LTERMP};  
(24,0)*{\LSLANT};  
(40,0)*{\RSLANT};  
(56,0)*{\LSLANT};  
(72,0)*{\RSLANT};  
(84,0)*{\EIGHTBOX};  
\endxy
}

\newcommand{\LXONEP}
{\xy
(-8,0)*{\TERMEXTP};  
(0,0)*{\TERMEXTP};  
(20,0)*{\LTERMP};  
(44,0)*{\LSLANT};  
(60,0)*{\RSLANT};  
(76,0)*{\LSLANT};  
(92,0)*{\RSLANT};  
\endxy
}

\newcommand{\LXZERON}
{
\xy
(-20,0)*{\TERMEXTN};  
(0,0)*{\LTERMN};  
(24,0)*{\RSLANT};  
(40,0)*{\LSLANT};  
(56,0)*{\RSLANT};  
(72,0)*{\LSLANT};  
(84,0)*{\EIGHTBOX};  
\endxy
}

\newcommand{\LXONEN}
{\xy
(-8,0)*{\TERMEXTN};  
(0,0)*{\TERMEXTN};  
(20,0)*{\LTERMN};  
(44,0)*{\RSLANT};  
(60,0)*{\LSLANT};  
(76,0)*{\RSLANT};  
(92,0)*{\LSLANT};  
\endxy
}

%
%

\subsubsection{Describing the elements}

We make use of the rectangular diagrams of \cite{CFP}.  In
\tref{Slants} below are the digrams for the basic building blocks
shown in \tref{BasicBlocks}.

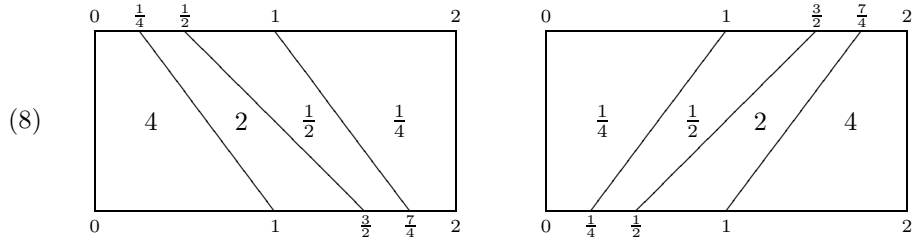
\mymargin{Slants}\begin{equation}\label{Slants}
\xy
(0,0)*{\BIGRSLANT};   (60,0)*{\BIGLSLANT};
\endxy
\end{equation}

The numbers at the top give coordinates in the domain, and the
numbers at the bottom give coordinates for the range.  The function
is viewed as going from the top of the rectangle to the bottom.  The
numbers in the middle give the slopes.  The slopes will be useful in
calculating the effects of some compositions.

The two figures in \tref{Slants} are mutual inverses.  Note that,
with the exceptions of the numbers across the middle, the the two
figures in \tref{Slants} are reflections of each other across a
horizontal line through the center.  The slopes in one figure are
the reciprocals of the slopes in the other.

To describe more complicated functions, we will put together smaller
versions of the pictures above, with less information about
coordinates and no information about slopes.  For example, the
function in \tref{SecondBlock} would be decribed by the following.

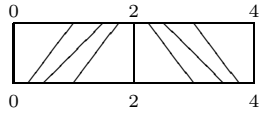
\[
\xy
(0,0)*{\LSLANT}; (16,0)*{\RSLANT};
(-8,6)*{\scriptstyle 0};
(8,6)*{\scriptstyle 2};
(24,6)*{\scriptstyle 4};
(-8,-6)*{\scriptstyle 0};
(8,-6)*{\scriptstyle 2};
(24,-6)*{\scriptstyle 4};
\endxy
\]

To describe the function in the right figure of \tref{LeftEnd} we
will make use of the diagram in the left of \tref{LTerms} below.
The right figure in \tref{LTerms} is the inverse of the left figure.

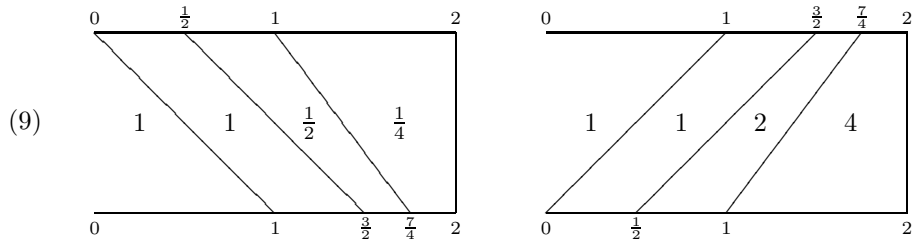
\mymargin{LTerms}\begin{equation}\label{LTerms}
\xy
(0,0)*{\BIGLTERMP};   (60,0)*{\BIGLTERMN};
\endxy
\end{equation}

The coordinates in \tref{LTerms} have been arbitrarily chosen to
start at 0, and the left edge is missing since 0 is not a fixed
point.  A diagram for the function in the right part of
\tref{LeftEnd} is as follows.  We do not bother with the coordinates
on the bottom.
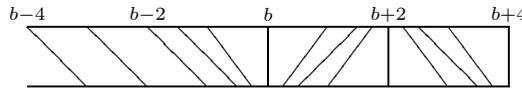
\[
\xy
(0,0)*{\LTERMP}; (24,0)*{\LSLANT}; (40,0)*{\RSLANT};
(-16,6)*{\scriptstyle b-4};
(0,6)*{\scriptstyle b-2};
(16,6)*{\scriptstyle b};
(32,6)*{\scriptstyle b+2};
(48,6)*{\scriptstyle b+4};
\endxy
\]

To describe the function in \tref{RightEnd}, we will use the left
figure in \tref{RTerms} below whose inverse is in the right part of
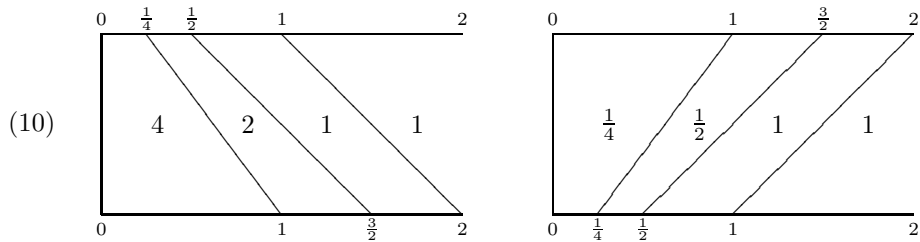
\tref{RTerms}.

\mymargin{RTerms}\begin{equation}\label{RTerms}
\xy
(0,0)*{\BIGRTERMP};   (60,0)*{\BIGRTERMN};
\endxy
\end{equation}

A diagram for the function in \tref{RightEnd} is as follows where we
show only a few coordinates.
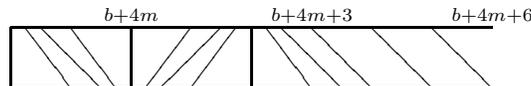
\[
\xy
(0,0)*{\RSLANT}; (16,0)*{\LSLANT}; (40,0)*{\RTERMP};
(8,6)*{\scriptstyle b+4m};
(32,6)*{\scriptstyle b+4m+3};
(56,6)*{\scriptstyle b+4m+6};
\endxy
\]

As is seen, coordinates such as \(b+4m\) and \(b+4m+6\) are
cumbersome.  From now on \(\xi\) will represent \(b+4m+2\), the
rightmost fixed point (which happens to be repelling) of \(X_0\).

The diagrams are too bulky to show all parts of a given element from
\tref{TheDefs}.  We will restrict ourselves to diagrams in the
neighborhood of \(b\) and diagrams in the neighborhood of \(\xi =
b+4m+2\).

\subsubsection{The analysis of supports near \protect\(\xi
\protect\)}\mylabel{XiSupSec}

Let us first tackle diagrams at the right end, in the neighborhood
of \(\xi\).  We start with the generators.

First we have \(X_0\).
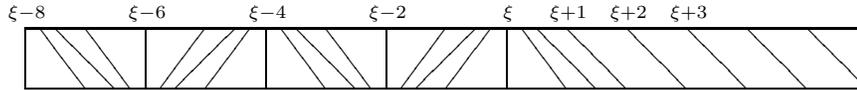
\[
\xy
(0,0)*{\XZEROP};
(27,6)*{\RXCOORDS};
\endxy
\]

Next we have \(X_1\).
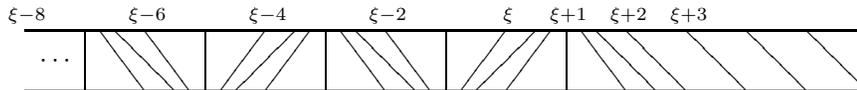
\[
\xy
(0,0)*{\XONEP};
(27,6)*{\RXCOORDS};
\endxy
\]

%
%
%
%

Their inverses are obtained by reflecting across a central horizontal
line.

Compositions are shown by stacking the diagrams vertically.  We
start with the simpler of \(S\) and \(T\).

Since \(S=X_1X_0X_1^{-1}X_1^{-1}\), we get the following diagram for
\(S\). 
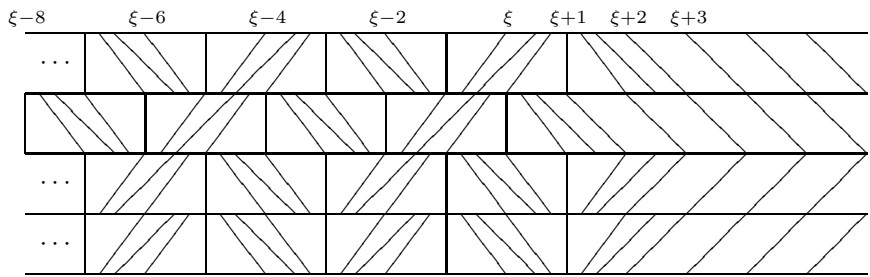
\[
\xy
(0,0)*{\XONEP};
(0,-8)*{\XZEROP};
(0,-16)*{\XONEN};
(0,-24)*{\XONEN};
(27,6)*{\RXCOORDS};
\endxy
\]

The picture above shows that \(\xi+1\frac14\) is an upper bound for
the support of \(S\).  The actual right endpoint for the support
needs more careful inspection.  By tracing the slopes from top to
bottom, using the figures in \tref{Slants} and \tref{RTerms}, it is
seen that between \(\xi+1\frac18\) and \(\xi+1\frac14\) the slopes
encountered are \(4\), \(1\), \(1\) and \(\frac14\) in that order
from top to bottom.  However, the slopes between \(\xi\) and
\(\xi+1\frac18\) are \(4\), \(1\), \(\frac12\) and \(\frac14\).
Thus \(\xi+1\frac18\) is the right endpoint of the support of \(S\).

Our next task is to tackle \(T= X_0 X_1 X_0^{-1}
X_0^{-1} X_1 X_0 X_1^{-1} X_1^{-1} X_0 X_0 X_1^{-1} X_0^{-1}\).
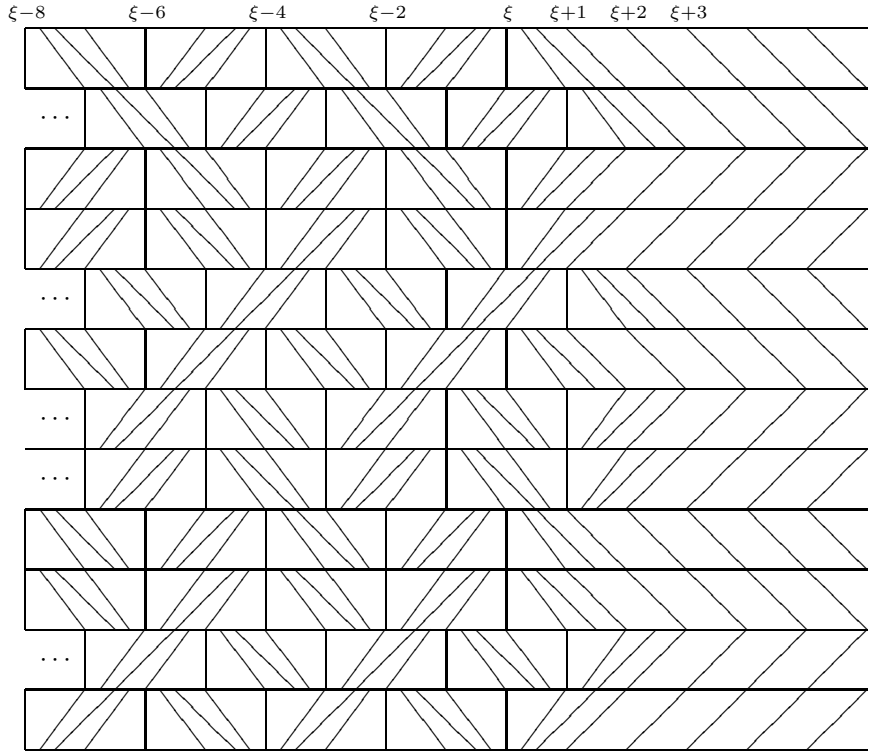
\[
\xy
(0,0)*{\XZEROP};
(0,-8)*{\XONEP};
(0,-16)*{\XZERON};
(0,-24)*{\XZERON};
(0,-32)*{\XONEP};
(0,-40)*{\XZEROP};
(0,-48)*{\XONEN};
(0,-56)*{\XONEN};
(0,-64)*{\XZEROP};
(0,-72)*{\XZEROP};
(0,-80)*{\XONEN};
(0,-88)*{\XZERON};
(27,6)*{\RXCOORDS};
\endxy
\]

In an analysis almost identical to that of \(S\), we get that
\(\xi+1\frac18\) is the right endpoint of the support of \(T\).

Now we look at \(C=X_0X_0X_1X_1X_0^{-1}X_0^{-1}X_1^{-1}X_1^{-1}\).
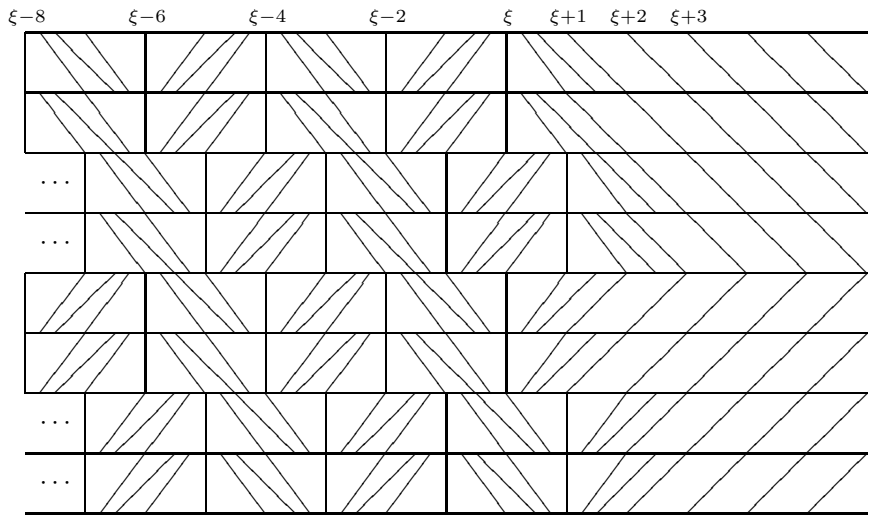
\[
\xy
(0,0)*{\XZEROP};
(0,-8)*{\XZEROP};
(0,-16)*{\XONEP};
(0,-24)*{\XONEP};
(0,-32)*{\XZERON};
(0,-40)*{\XZERON};
(0,-48)*{\XONEN};
(0,-56)*{\XONEN};
(27,6)*{\RXCOORDS};
\endxy
\]

We get the following from the picture above.  For an integer \(j>0\)
for which \(\xi-4j\) is in the pattern above, we have
\(\xi-4j+\frac14\) is carried by \(C\) to at least
\(\xi-4j+4\frac34\).  In particular \(\xi-4+\frac14\) is carried to
greater than \(\xi+\frac34\).  Using the information in
\tref{RTerms} with the figure above, we get that \(\xi+\frac34\) is
carried to \(\xi+1\frac3{16}\). Further, the interval from \(\xi+1\)
to \(\xi+1\frac12\) is carried affinely with slope \(\frac12\) to
the interval from \(\xi+1\frac14\) to \(\xi+1\frac12\).  Thus
\(\xi+1\frac12\) is the right endpoint of the support of \(C\).

It follows that for an integer \(j>0\) for which \(\xi-4j\) is in
the pattern above, we have that any \(\eta\) in \((\xi-4j + \frac14,
\xi+1\frac12)\) has \(\eta C>\eta\).  Further, any such \(\eta\) has
\mymargin{CAtRight}\begin{equation}\label{CAtRight}
\begin{split} \eta C^j &> \xi+{\textstyle \frac34}, \\
\eta C^{j+1} &> \xi+1{\textstyle \frac3{16}}. \end{split}
\end{equation}

The point is that \(\xi+1\frac3{16}\) is greater than
\(\xi+1\frac18\), the right endpoint of the support of \(S\) and of
\(T\).

\subsubsection{The analysis of supports near \protect\(b
\protect\)}\mylabel{BSupSec}

We create pictures for the generators near \(b\) in much the same
way as we do near \(\xi\).  The reader can verify that the that the
following is an accurate combination of diagrams for \(X_0^{\pm1}\)
and \(X_1^{\pm1}\) that gives the behavior of
\(C=X_0X_0X_1X_1X_0^{-1}X_0^{-1}X_1^{-1}X_1^{-1}\) near \(b\).
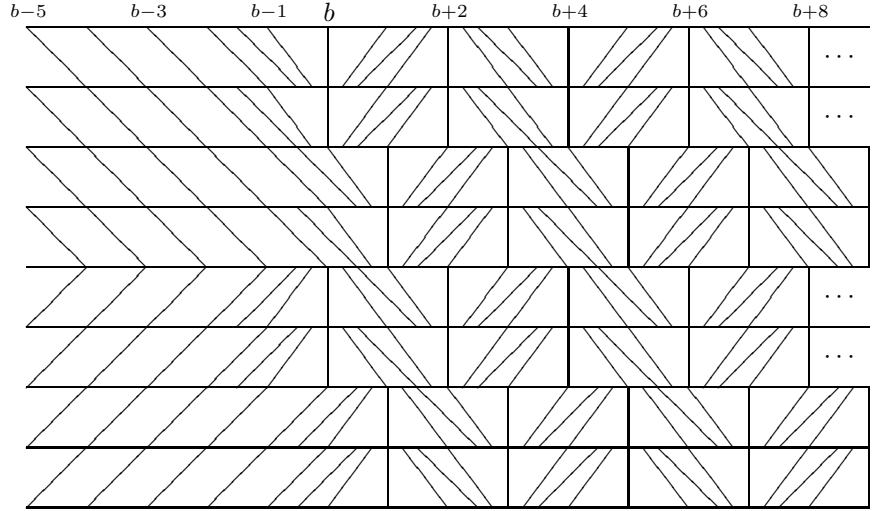
\[
\xy
(0,0)*{\LXZEROP};
(0,-8)*{\LXZEROP};
(0,-16)*{\LXONEP};
(0,-24)*{\LXONEP};
(0,-32)*{\LXZERON};
(0,-40)*{\LXZERON};
(0,-48)*{\LXONEN};
(0,-56)*{\LXONEN};
(27,6)*{\LXCOORDS};
\endxy
\]

From the diagram above, we get the following information.  First, the
left endpoint of the support of \(C\) near \(b\) is \(b-4\frac12\).
Second, a value of \(j\) for which \tref{CAtRight} is valid is that
\(j\) for which \(\xi-4j=b+2\).  This value of \(j\) satisfies
\(b+4m+2-4j=b+2\) or \(j=m\).  Third, we note that 
\((b-3\frac12)C > b+2\frac34\).  Combining this information with
\tref{CAtRight}, we get that for any \(\eta\ge(b-3\frac12)\) we have 
\mymargin{CAtLeft}\begin{equation}\label{CAtLeft}
\eta C^{m+2} > {\textstyle \xi+1\frac3{16}}.
\end{equation}

We next look at \(S\) and \(T\).

The diagram for \(S=X_1X_0X_1^{-1}X_1^{-1}\) near \(b\) follows.
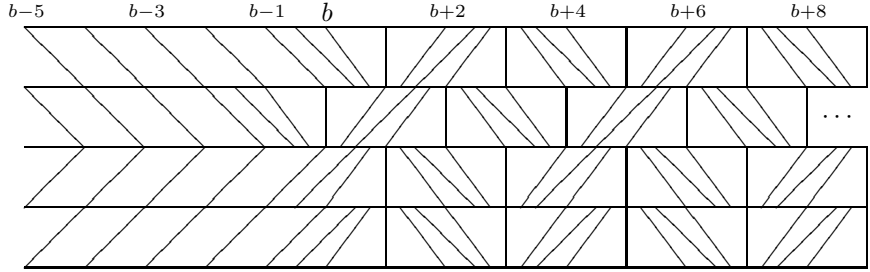
\[
\xy
(0,0)*{\LXONEP};
(0,-8)*{\LXZEROP};
(0,-16)*{\LXONEN};
(0,-24)*{\LXONEN};
(27,6)*{\LXCOORDS};
\endxy
\]

The diagram for \(T= X_0 X_1 X_0^{-1}
X_0^{-1} X_1 X_0 X_1^{-1} X_1^{-1} X_0 X_0 X_1^{-1} X_0^{-1}\) near
\(b\) is below.
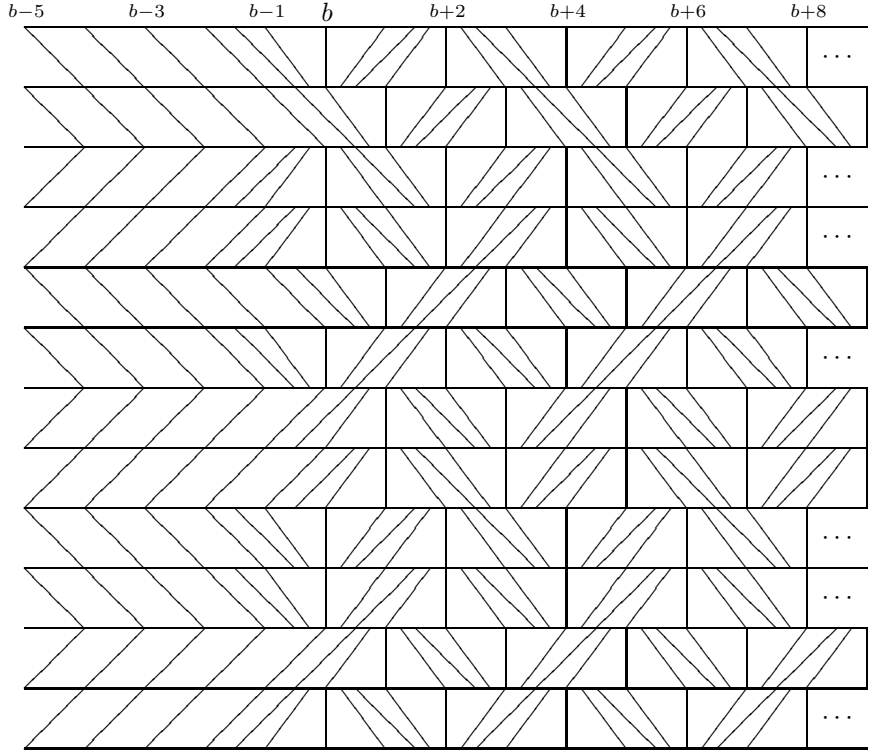
\[
\xy
(0,0)*{\LXZEROP};
(0,-8)*{\LXONEP};
(0,-16)*{\LXZERON};
(0,-24)*{\LXZERON};
(0,-32)*{\LXONEP};
(0,-40)*{\LXZEROP};
(0,-48)*{\LXONEN};
(0,-56)*{\LXONEN};
(0,-64)*{\LXZEROP};
(0,-72)*{\LXZEROP};
(0,-80)*{\LXONEN};
(0,-88)*{\LXZERON};
(27,6)*{\LXCOORDS};
\endxy
\]

A trace through the two diagrams above, using the information in
\tref{LTerms} about the slopes, shows that the left endpoint of the
supports of \(S\) and \(T\) near \(b\) is \(b-2\frac12\).

The fact that the left endpoint of the support of \(C\) near \(b\)
is \((b-4\frac12)\) and that the left endpoint of the supports of
both \(S\) and \(T\) near \(b\) is \((b-2\frac12)\) explains why our
claims in Section \ref{Outline} refer to \([0,b-5)\) and \((b-5,
\infty)\).  

\subsubsection{End of the proof of Part (i)}

From Sections \ref{XiSupSec} and \ref{BSupSec}, we have the
following information.  Using the fact that \(\xi=b+4m+2\), we
have that the supports of \(C\), \(S\) and \(T\) in
\((b-5, \infty)\) are given by 
\[
\begin{split}
C:\quad &{\textstyle (b-4\frac12 ,\,\, b+4m+3\frac12)}, \\
S,\,T:\quad &{\textstyle (b-2\frac12 ,\,\, b+4m+3\frac18)}.
\end{split}
\]

From \tref{CAtLeft} and the fact that \(i=m+2\), we know that for
any \(\eta \ge (b-3\frac12)\) 
we have 
\[
\eta C^i > {\textstyle b+4m+3\frac3{16}}.
\]

From the facts above, we know that the supports of
\(\Sigma=C^{-i}SC^i\) and \(\Theta=C^{-i}TC^i\) in \((b-5, \infty)\)
are both contained in \[({\textstyle b+4m+3\frac3{16}, \,\,
b+4m+3\frac12})\] which is disjoint from the supports of \(S\) and
\(T\) in \((b-5, \infty)\).  Thus the restrictions of \(Z = [S,\Sigma]\)
and \(W=[T,\Theta]\) to \((b-5, \infty)\) are trivial.

From our discussion in Section \ref{FirstLookSec}, we see that
\(W=w\) and \(Z=z\).  From the definition \(P=Z^{-1}W\), we get
\(P=p\).  Lastly, with \(Q\), \(H\) and \(K\) defined in terms of
\(P\) and \(X_1\), we get the rest of (i) of Proposition
\ref{PROP}.

\subsection{Proposition \ref{PROP} Part (iv)}

We must show that the translates of the interval \((1,3)\) under
words in \(X_0\) and \(X_1\) cover all of \((0,\infty)\).

Since the orbit of \((b-1)\) under \(X_0\) includes all the integers
below \(b-1\) as well as all fractions of the form \(1/2^n\) for a
positive integer \(n\), we get that the translates cover at least
\((0,b-1)\).  

Since \((1, b-1)\) is covered by finitely many translates of
\((1,3)\), we may work from now on with \((1, b-1)\) instead of
\((1,3)\).

Since \(C\) has a fixed point at \((b-5)\) which we take to be
bigger than 1, and since \((b-1)\) is carried by powers of \(C\) to
at least \((b+4m+3\frac12)\).  We can now make another replacement
and work from now on with the interval \((1, b+4m+3)\).

From the discussion above \tref{RightEnd}, we know that \(X_0\) acts
as translation by 1 on \([b+4m+3, \infty)\).  Since \(X_0\) has a
fixed point at \((b+4m+2)\), we get all points in \((1,\infty)\)
covered.

Combining the information in the four paragraphs above completes
the argument for (iv).


\begin{thebibliography}{1}

\bibitem{AST}
Azer Akhmedov, Melanie Stein and Jennifer Taback, 
\emph{Free limits of {T}hompson's group F}, announcement, conference on
Geometric and Asymptotic Group Theory with Applications, March 9--12, 2009, 
Stevens Institute, Hoboken, NJ.

\bibitem{brin+fer}
Matthew~G. Brin and Fernando Guzm{\'a}n, \emph{Automorphisms of generalized
  {T}hompson groups}, J. Algebra \textbf{203} (1998), no.~1, 285--348.

\bibitem{brin+squier}
Matthew~G. Brin and Craig~C. Squier, \emph{Groups of piecewise linear
  homeomorphisms of the real line}, Invent. Math. \textbf{79} (1985), 485--498.

\bibitem{brown:finiteprop}
Kenneth~S. Brown, \emph{Finiteness properties of groups}, Journal of Pure and
  Applied Algebra \textbf{44} (1987), 45--75.

\bibitem{CFP}
J.~W. Cannon, W.~J. Floyd, and W.~R. Parry, \emph{Introductory notes on
  {R}ichard {T}hompson's groups}, Enseign. Math. (2) \textbf{42} (1996),
  no.~3-4, 215--256. \MR{98g:20058}

\bibitem{MR2151593}
Christophe Champetier and Vincent Guirardel, \emph{Limit groups as limits of
  free groups}, Israel J. Math. \textbf{146} (2005), 1--75. \MR{MR2151593
  (2006d:20045)}

\bibitem{esyp}
Evgenii~S. Esyp, \emph{On identities in {T}hompson's group}, ArXiv preprint:
  http://front.math.ucdavis.edu/0902.0199, 2009.

\bibitem{Sapir+olshanskii:free-like}
M.~V. Sapir and A.~Yu. Olshanskii, \emph{On k-free-like groups}, ArXiv
  preprint: http://front.math.ucdavis.edu/0811.1607, 2008.

\bibitem{MR1863735}
Zlil Sela, \emph{Diophantine geometry over groups. {I}. {M}akanin-{R}azborov
  diagrams}, Publ. Math. Inst. Hautes \'Etudes Sci. (2001), no.~93, 31--105.
  \MR{MR1863735 (2002h:20061)}

\bibitem{zarzycki}
Roland Zarzycki, \emph{Limits of ({T}hompson's) group {F}}, ArXiv preprint:
  http://front.math.ucdavis.edu/0701.5601, 2007.

\end{thebibliography}

\providecommand{\bysame}{\leavevmode\hbox to3em{\hrulefill}\thinspace}
\providecommand{\MR}{\relax\ifhmode\unskip\space\fi MR }
\providecommand{\MRhref}[2]{%
  \href{http://www.ams.org/mathscinet-getitem?mr=#1}{#2}
}
\providecommand{\href}[2]{#2}

\noindent Department of Mathematical Sciences

\noindent State University of New York at Binghamton

\noindent Binghamton, NY 13902-6000

\noindent USA

\noindent email: matt@math.binghamton.edu

\end{document}